\documentclass[10pt]{amsart}
\usepackage{amsmath}
\usepackage{amscd,amsthm,amssymb,amsfonts}
\usepackage{mathrsfs}
\usepackage{dsfont}
\usepackage{stmaryrd}
\usepackage{euscript}
\usepackage{expdlist}
\usepackage{enumerate}

\input xy
\xyoption{all}
\usepackage[OT2,T1]{fontenc}

\theoremstyle{plain}
\newtheorem{thm}{Theorem}[section]

\newtheorem*{thm*}{Theorem}

\newtheorem{lm}[thm]{Lemma}
\newtheorem{cor}[thm]{Corollary}
\newtheorem*{cor*}{Corollary}

\newtheorem{prop}[thm]{Proposition}
\newtheorem*{conj*}{Conjecture}

\theoremstyle{remark}
\newtheorem*{remark}{Remark}
\newtheorem*{thank}{Acknowledgments}

\theoremstyle{definition}
\newtheorem*{defn*}{Definition}

\newcommand{\nc}{\newcommand}

\newcommand{\beq}{\begin{equation}}
\newcommand{\eeq}{\end{equation}}
\newcommand{\bpmx}{\begin{pmatrix}}
\newcommand{\epmx}{\end{pmatrix}}
\newcommand{\bbmx}{\begin{bmatrix}}
\newcommand{\ebmx}{\end{bmatrix}}

\newcommand{\wtd}{\widetilde}

\newcommand{\beqcd}[1]{\begin{equation*}\label{#1}\tag{#1}}
\newcommand{\eeqcd}{\end{equation*}}

\numberwithin{equation}{section}

\def\makeop#1{\expandafter\def\csname#1\endcsname
  {\mathop{\rm #1}\nolimits}\ignorespaces}

\makeop{Hom}   \makeop{End}   \makeop{Aut}   
\makeop{Pic} \makeop{Gal}       \makeop{Div} \makeop{Lie}
\makeop{PGL}   \makeop{Corr} \makeop{PSL} \makeop{sgn} \makeop{Spf}
 \makeop{Tr} \makeop{Nr} \makeop{Fr} \makeop{disc}
\makeop{Proj} \makeop{supp} \makeop{ker}   \makeop{Im} \makeop{dom}
\makeop{coker} \makeop{Stab} \makeop{SO} \makeop{SL} \makeop{SL}
\makeop{Cl}    \makeop{cond} \makeop{Br} \makeop{inv} \makeop{rank}
\makeop{id}    \makeop{Fil} \makeop{Frac}  \makeop{GL} \makeop{SU}
\makeop{Trd}   \makeop{Sp} \makeop{Tr}    \makeop{Trd} \makeop{Res}
\makeop{ind} \makeop{depth} \makeop{Tr} \makeop{st} \makeop{Ad}
\makeop{Int} \makeop{tr}    \makeop{Sym} \makeop{can} \makeop{SO}
\makeop{torsion} \makeop{GSp} \makeop{Tor}\makeop{Ker} \makeop{rec}
\makeop{Ind} \makeop{Coker}
 \makeop{vol} \makeop{Ext} \makeop{gr} \makeop{ad}
 \makeop{Gr}\makeop{corank} \makeop{Ann}
\makeop{Hol} 
\makeop{Fitt} \makeop{Mp} \makeop{CAP}
\def\Sel{Sel}

\DeclareMathAlphabet{\mathpzc}{OT1}{pzc}{m}{it}
\DeclareSymbolFont{cyrletters}{OT2}{wncyr}{m}{n}
\DeclareMathSymbol{\SHA}{\mathalpha}{cyrletters}{"58}

\def\makebb#1{\expandafter\def
  \csname bb#1\endcsname{{\mathbb{#1}}}\ignorespaces}
\def\makebf#1{\expandafter\def\csname bf#1\endcsname{{\bf
      #1}}\ignorespaces}
\def\makegr#1{\expandafter\def
  \csname gr#1\endcsname{{\mathfrak{#1}}}\ignorespaces}
\def\makescr#1{\expandafter\def
  \csname scr#1\endcsname{{\EuScript{#1}}}\ignorespaces}
\def\makecal#1{\expandafter\def\csname cal#1\endcsname{{\mathcal
      #1}}\ignorespaces}

\def\doLetters#1{#1A #1B #1C #1D #1E #1F #1G #1H #1I #1J #1K #1L #1M
                 #1N #1O #1P #1Q #1R #1S #1T #1U #1V #1W #1X #1Y #1Z}
\def\doletters#1{#1a #1b #1c #1d #1e #1f #1g #1h #1i #1j #1k #1l #1m
                 #1n #1o #1p #1q #1r #1s #1t #1u #1v #1w #1x #1y #1z}
\doLetters\makebb   \doLetters\makecal  \doLetters\makebf
\doLetters\makescr
\doletters\makebf   \doLetters\makegr   \doletters\makegr

\normalsize

\makeop{Ram} \makeop{Rep} \makeop{mass}

\makeop{Bl}
\def\abs#1{\left|#1\right|}

\def\Qbarp{\C_p}
\def\Qp{\Q_p}

\def\Zbarp{\ol{\Z}_p}

\def\Zp{\Z_p}

\def\Qbar{\ol{\Q}}

\def\cD{\mathcal D}
\def\cE{{\mathcal E}}
\def\cF{{\mathcal F}}  

\def\cL{{\mathcal L}}

\def\cK{{\mathcal K}}  

\def\cR{{\mathcal R}}
\def\cO{\mathcal O}

\def\cf{{\mathcal f}}

\def\cX{\mathcal X}

\def\cU{\mathcal U}

\def\bfc{\mathbf c}

\def\bdc{\mathbf c}

\def\bftheta{\boldsymbol{\theta}}

\def\bbI{\mathbb I}

\newcommand{\Z}{\mathbf Z}
\newcommand{\Q}{\mathbf Q}
\newcommand{\R}{\mathbf R}
\newcommand{\C}{\mathbf C}
\newcommand{\A}{\mathbf A}    

\def\bbE{{\mathbb E}}

\def\cond{{c}}

\def\frakc{{\mathfrak c}}

\def\frakm{\mathfrak m}

\def\frakF{{\mathfrak F}}

\def\frakD{\mathfrak D}

\def\frakC{{\mathfrak C}}

\def\zbar{\bar{z}}

\def\isoto{\stackrel{\sim}{\to}}

\def\ot{\otimes}

\def\hookto{\hookrightarrow}
\def\longto{\longrightarrow}
\def\ol{\overline}  \nc{\opp}{\mathrm{opp}} \nc{\ul}{\underline}

\def\cf{\mbox{{\it cf.} }}

\def\uf{\varpi} 
\def\Abs{{|\!\cdot\!|}} 

\def\Sg{{\varSigma}}  
\def\ndivides{\nmid}
\def\ndivide{\nmid}
\def\x{{\times}}

\def\onehalf{{\frac{1}{2}}}

\def\con{\equiv}
\def\bksl{\backslash}
\newcommand\stt[1]{\left\{#1\right\}}

\def\lam{\lambda}

\def\AFf{\A_{\cF,f}}
\def\AKf{\A_{\cK,f}}
\def\AF{\A_\cF}
\def\AK{\A_\cK}
\def\setp{{(p)}}

\newcommand{\powerseries}[1]{\llbracket{#1}\rrbracket}

\renewcommand\pmod[1]{\,(\mbox{mod }{#1})}
\renewcommand\mod[1]{\,\mbox{mod }{#1}}

\usepackage[utf8x]{inputenc}
\usepackage{amsmath}
\usepackage{amssymb}
\usepackage{amsthm}
\usepackage{amsmath,amscd}
\title{On the $\mu$-invariant of the cyclotomic derivative \linebreak of Katz $p$-adic L-function\\
}
\author{Ashay A. Burungale}
\date{\today}
\address{ Department of Mathematics~\\UCLA ~ \\
Los Angeles, CA 90095-1555, USA
}
\email{ashayburungale@gmail.com}
\def\Csplit{\frakF}
\def\holES{\bbE^h}

\def\padic{\text{$p$-adic }}

\begin{document}
\maketitle

\begin{abstract}
When the branch character has root number $-1$, the corresponding anticyclotomic Katz $p$-adic \\
L-function identically vanishes. In this case, we determine the $\mu$-invariant of the cyclotomic derivative of Katz $p$-adic L-function.  
As an application, the result proves the non-vanishing of the anticyclotomic regulator of a self-dual CM modular form with root number
$-1$. The result also plays a crucial role in the recent work of Hsieh on the Eisenstein ideal approach to a one-sided divisibility of the CM main conjecture.\\
\\
\end{abstract}
\tableofcontents

\section{ Introduction}
\noindent Zeta values often enter the $p$-adic world via $p$-adic L-functions. One expects that $p$-adic L-functions are intimately
connected with arithmetic. A $p$-adic L-function can have several variables. In such a case, we may expect that a power series
obtained by taking its partial derivative with respect to one of the variables at a specific value of that variable, also has some 
arithmetic meaning.\\
\\
Katz $p$-adic L-function over a totally real field of degree $d$ has $(d+1+\delta)$-variables, where $\delta$ is the Leopoldt defect
for the totally real field. In this article, by Katz $p$-adic L-function we mean the projection to the
first $(d+1)$-variables i.e. the anticyclotomic and cyclotomic variables. When the branch character is self-dual with the
root number $-1$, the corresponding anticyclotomic Katz $p$-adic L-function of $d$-variables identically vanishes. In this article,
we study the $\mu$-invariant of the cyclotomic derivative of the Katz $p$-adic L-function when the branch character is of this type. Following a strategy of Hida, we determine this $\mu$.\\
\\
Let us introduce some notation. Fix an odd prime $p$. Let $\cF$ be a totally real field of degree $d$ over $ \Q$ and $\cO_{\cF}$ be the ring of integers. 
Let $\cK$ be a totally imaginary quadratic extension of $\cF$.
Let $D_{F}$ be the discriminant of $\cF/ \Q$. Fix two embeddings $\iota_{\infty}\colon \Qbar \to \C$ and $\iota_{p}\colon \Qbar \to \Qbarp$. Let $c$ denote the complex conjugation on $\C$ which induces the unique non-trivial element of $\Gal(\cK/\cF)$ via $\iota_{\infty}$. We assume the following hypothesis throughout:\\
\\
{$(ord)$}  \text{Every prime of $\cF$ above $p$ splits in $\cK$.}\\
\\
The condition $(ord)$ guarantees the 
existence of a $p$-adic CM type $\Sigma$ i.e. $\Sigma$ is a CM type of $\cK$ such that, $p$-adic places induced by elements
in $\Sigma$ via $\iota_{p}$ are disjoint from those induced by $\Sigma c$. 
Let $\cK_{\infty}^{+}$ and $\cK_{\infty}^{-}$ be the cyclotomic $\Zp$-extension and anticyclotomic $\Zp^{d}$-extension of $\cK$. 
Let $\cK_{\infty}=\cK_{\infty}^{+} \cK_{\infty}^{-}$ be a $\Zp^{d+1}$-extension of $\cK$.  Let $\Gamma^{\pm}:=\Gal(\cK^{\pm}_{\infty}/\cK)$ and let $\Gamma=\Gal(\cK_{\infty}/\cK) \simeq \Gamma^{+} \times \Gamma^{-}$.\\
\\
Let $\mathfrak{C}$ be a prime-to-$p$ integral ideal of $\cK$. Decompose $\mathfrak{C}=\mathfrak{C^{+}} \mathfrak{C^{-}}$,
where $\mathfrak{C^{+}}$ (respectively $\mathfrak{C^{-}}$) is a product of split primes (respectively ramified or inert primes) over $\cF$.
Let $\lambda$ be a Hecke character of infinity type $k\Sigma$, $k>0$ and
suppose that $\mathfrak{C}$ is the prime-to-$p$ conductor of $\lambda$. 
The Hecke character $\lam$ is $\overline{\Q}$-valued; we still denote by $\lam$ its composition with $\iota_p$ or $\iota_\infty$; the context will make 
it clear which is understood. 
Associated to this data, a $(d+1)$-variable Katz $p$-adic L-function
$L_{\Sigma,\lambda}(T_{1},T_{2}, ..., T_{d},S)\in \Zbarp \powerseries{\Gamma}$ is constructed in \cite{Ka} and \cite{HT}. Here $T_{1}, ..., T_{d}$ are the anticyclotomic
variables and $S$ is the cyclotomic variable. We occasionally abbreviate this function as $L_{\Sigma,\lambda}$. 
It interpolates critical Hecke L-values $L(0,\lam\chi)$ as $\chi$ varies over certain Hecke characters mod $\mathfrak{C}p^{\infty}$ 
(cf. \cite[Thm. II]{HT}). Let $L_{\Sigma,\lambda}^{-} \in \Zbarp\powerseries{\Gamma^{-}}$ be the anticyclotomic projection obtained by 
substituting $S=0$.\\
\\
For each local place $v$ of $\cF$, choose a uniformiser $\varpi_{v}$ and let $|\cdot|_{v}$ denote the corresponding absolute value normalised so that 
$|\varpi_v|_{v}=|N(\varpi_v)|_l$ and $|l|_{l}=\frac{1}{l}$, where 
$N$ is the norm, $v \cap \Q = (l)$ and $l > 0$. Let $v_{p}$ be the $p$-adic valuation of $\C_p$ normalised such that $v_{p}(p)=1$. 
We view it as a function on $\Qbar$ via $\iota_{p}$. 
Let $N$ be the norm Hecke character i.e. the adelic realisation of the $p$-adic cyclotomic character. 
For each $v$ dividing $\mathfrak{C^{-}}$ and $\lambda$ as above,
the local invariant $\mu_{p}(\lambda_{v})$ is defined by\\
\beq\label{E} \mu_{p}(\lambda_{v}) = \inf_{x \in \cK_{v}^{\times}} v_{p}(\lambda_{v}(x) - 1).\eeq
Let us also define 
\beq\label{E} \mu_{p,v}'(\lam) = v_{p}(\frac{\log_{p}(|\varpi_{v}|_{v})}{\log_{p}(1+p)}) + \sum_{w \neq v, w|\mathfrak{C^{-}}}^{} \mu_{p} (\lam_{w}) \eeq
and \beq\label{E}\mu_{p}'(\lam) = \sum_{v|\mathfrak{C^{-}}}^{} \mu_{p}(\lam_{v}).\eeq
From now on, suppose that $\lambda$ is self-dual i.e. $\lambda|_{\bf{A}_{\cF}^{\times}}=\tau_{\cK/\cF}|\cdot|_{\bf{A}_{\cF}}$, where $\tau_{\cK/\cF}$
is the quadratic character associated to $\cK/\cF$ and $|\cdot|_{\bf{A}_{\cF}}$ is the adelic norm. 
In particular, the root number of $\lambda$ is $\pm1$. When the root number equals one, the anticyclotomic $\mu$-invariant $\mu(L_{\Sg,\lam}^{-})$ has been studied 
in \cite{Hi} and \cite{Hs1}. 
Now suppose that the global root number is $-1$.
In view of the functional equation of Hecke L-function, this root number condition forces all the Hecke L-values appearing in the interpolation
property of $L_{\Sigma,\lambda}^{-}$ to vanish. Accordingly, $L_{\Sigma,\lambda}^{-}=0$. This also follows from the functional 
equation of $L_{\Sigma,\lam}$ (cf. \cite[\S 5]{HT}). The anticyclotomic arithmetic information contained
in $L_{\Sigma,\lambda}$ may seem to have disappeared. However, we can look at the cyclotomic derivative 
\beq L_{\Sigma,\lambda}^{'}= (\frac{\partial}{\partial S} L_{\Sigma, \lam}(T_{1},...,T_{d},S) )|_{S=0} .\eeq
In what follows our meaning of derivative is the following. 
Let $f$ be a function from integers to a $p$-adic domain of characteristic different from $p$. We define
\beq\label{E} \frac{d}{ds}f(s)|_{s=0} := \lim_{n \rightarrow \infty} \frac { f(p^{n}) - f(0)} {p^{n}}. \eeq
Here $\lim$ denotes the $p$-adic limit. Note that the Leibnitz product rule is valid for this notion of derivative.\\
\\
For any integer $k$, $L_{\Sigma,\lambda}(T_{1},T_{2}, ..., T_{d}, (1+p)^{k}-1)$ equals $L_{\Sigma, \lambda N^{k}}^{-}$ and 
$\lim_{n \rightarrow \infty} ( (1+p)^{p^{n}}-1 )$ equals zero.
Thus, the cyclotomic derivative equals
\beq\label{E} \frac{1}{\log_{p}(1+p)}(\frac{d}{ds} L_{\Sigma, \lambda N^{s}}^{-} )|_{s=0} \in \Zbarp\powerseries{\Gamma^{-}}.\eeq
Note that the factor $\frac{1}{\log_{p}(1+p)}$ comes from the fact that
$$ \lim_{n \rightarrow \infty} \frac { (1+p)^{p^{n}} - 1} {p^{n}}= \log_p(1+p).$$
\\
Here $\log_p$ is Iwasawa's $p$-adic logarithm normalised so that $\log_p(p)=0$.\\
\\
The following is perhaps the main result of the article.\\
\\
\\
{\bf Theorem A} Let $h_\cK^-:=h_\cK/h_\cF$ be the relative class number. Suppose that $p\ndivide h_\cK^-\cdot D_\cF$. Then, we have\\
$$\mu(L_{\Sigma,\lam}^{'}) = \min_{v|\mathfrak{C^{-}}} \{\mu_{p}'(\lam), \mu_{p,v}'(\lam)\}.$$
\\
\\
We now describe the strategy of the proof. Some of the notation used here is not followed in the rest
of the article.\\
\\
We basically follow a strategy of Hida. Let us briefly recall Hida's strategy to determine $\mu(L_{\Sigma,\lam}^{-})$ (cf. \cite{Hi}).
Suppose that $p\ndivide h_\cK^-\cdot D_\cF$. Let $O_{p}=\cO_{\cF} \otimes \mathbb{Z}_p$. 
Let $G_{\Sigma,\lam} \in \Zbarp \powerseries{T_1,...,T_d}$ be the power series expansion of the 
measure $L_{\Sigma,\lam}^-$ regarded as a $p$-adic measure on $O_p$ with support in $1+pO_p$, given by 
\beq G_{\Sigma,\lam}= \int_{1+pO_p} t^{y} dL_{\Sigma,\lam}^{-}(y) = \sum_{(k_1,...,k_d)\in (\Z_{\geq 0})^{d}} \Big{(}\int_{1+pO_p} \binom {y}{k_1,...,k_d} dL_{\Sigma,\lam}^{-}(y) \Big{)} T_1^{k_1}...T_{d}^{k_d}. \eeq
The starting point is the observation that there are classical Hilbert modular Eisenstein series $(f_{\lam,i})_i$ such that
\beq 
G_{\Sigma,\lam}=\sum_{i}a_i\circ(f_{\lam,i}(t)),
\eeq
where $f_{\lam,i}(t)$ is the
$t$-expansion of $f_{\lam,i}$ around a well chosen CM point $x$ with the CM type $(\cK,\Sigma)$ 
on the Hilbert modular Shimura variety $Sh$ and $a_i$ is an automorphism of the deformation space of $x$ in $Sh$. 
Based on Chai's study of Hecke-stable subvarieties of a Shimura variety, Hida has proven the linear independence of $(a_i\circ f_{\lam,i})_i$
modulo $p$. It follows that $\mu(L_{\Sigma,\lam}^{-})=\min_{i} \mu(f_{\lam,i}(t))=\min_{i}\mu(f_{\lam,i})$. 
Now $\mu(f_{\lam,i})=\mu(f_{\lam,i}(q))$, where $f_{\lam,i}(q)$ is the $q$-expansion of $f_{\lam,i}$. 
Thus, the question reduces to the computation of the $\mu$-invariant of the 
Fourier expansion of $f_{\lam,i}$ by the $q$-expansion principle. When $\mathfrak{C}^{-}=1$, this computation can be done quite explicitly. 
However, when $\mathfrak{C}^{-}\neq1$, the computation seems quite complicated.
As an alternative, Hsieh has introduced certain Hilbert modular Eisenstein series $({\bf{f}}_{\lam,i})_i$ 
called toric Eisenstein series 
whose
$q$-expansion compuation is a bit simpler than that of $(f_{\lam,i})_i$ such that the property (1.8) still holds i.e.
the power series $G_{\Sigma,\lam}$ equals $\sum_{i}a_i\circ({\bf{f}}_{\lam,i}(t))$
(cf. \cite{Hs1}). The Fourier coefficients of these toric Eisenstein series admit a product decomposition of local Whittaker integrals. In \cite{Hi} and \cite{Hs1}, the condition
$p\ndivide h_\cK^-$ is not needed as in general the power series $L_{\Sigma,\lam}^{-}$ restricted to
an explicit finite open cover is still of the form (1.8). 
However, to obtain an expression for the cyclotomic derivate $L_{\Sg,\lam}^{'}$ of the form (1.8), the condition seems to be necessary.\\
\\
In our case, the root number condition implies that $\mathfrak{C^{-}} \neq 1$. So, we use the toric Eisenstein series. 
Let $G_{\Sigma,\lam}' \in \Zbarp \powerseries{T_1,...,T_d}$ be the power series expansion of $L_{\Sigma,\lam}'$ defined as in (1.8). 
We show that there are $p$-adic Hilbert modular forms $({\bf{f}}_{\lam',i})_i$ such that
\beq
G_{\Sigma,\lam'}=\sum_{i}a_i\circ({\bf{f}}_{\lam',i}(t)).
\eeq
Basically, ${\bf{f}}_{\lam',i}$ is the derivative of ${\bf{f}}_{\lam N^s,i}$ at $s=0$ (cf. (1.5)).
Via Hida's strategy, the question then reduces to the computation of the $\mu$-invariant of $q$-expansion of $({\bf{f}}_{\lam',i})_i$.
However, the expression for the $q$-expansion coefficients does not seem to be explicit.
From the vanishing of the toric Eisenstein series associated to $\lam$ and the Leibnitz product rule, the coefficients admit a product decomposition of local Whittaker integrals and derivative. 
We obtain an expression for the derivative. Based on \cite{Hs1} and \cite{Hs2}, 
we analyse the non-triviality of the local Whittaker integrals and derivative modulo $p$ and 
by patching determine the non-triviality of ${\bf{f}}_{\lam',i}$ modulo $p$.\\
\\
When $\cF$ equals $\Q$ and $\lam$ is the Gr\"{o}ssencharacter associated to a CM elliptic curve $E/\Q$ having CM by $\cO_K$, 
the Katz $p$-adic L-function, $L_{\Sigma,\lam}$ is the two variable commutative $p$-adic L-function associated to $E$. 
In \cite {Ru}, Rubin proves that $L_{\Sigma,\lam}$ generates the characteristic ideal of a certain Selmer group associated to $E/\cK_\infty$.
This two variable main conjecture gives cyclotomic and anticyclotomic main conjectures. 
When $\lam$ has root number $-1$, both sides of the anticyclotomic main conjecture are zero (cf. \cite{AH}). 
However, in [loc. cit., Thm A] it is shown that $L_{\Sigma,\lam}^{'}$ generates the characteristic ideal of the torsion part of the 
anticyclotomic Selmer group times a certain anticyclotomic regulator associated to $\lam$ after tensoring with $\Q_p$. 
In the appendix of \cite{AH}, Rubin proves the non-vanishing of this anticyclotomic regulator. Thus, $L_{\Sigma,\lam}^{'}$ is non-trivial.
The results of [loc. cit.] have been generalised to self-dual CM modular forms in \cite{A}, 
except the non-vanishing of the anticyclotomic regulator. 
Theorem A proves the non-vanishing. 
This seems to be one of the first instances where the
non-vanishing of an Iwasawa theoretic regulator is proven by a modular method (combined with the main conjecture).\\
\\
We now make brief remarks regarding the role in the one-sided divisibility of the CM main conjecture for Hecke characters. We refer to the introduction and \S8 of \cite{Hs3} for the details. 
In \cite{Hs3}, Hsieh constructs a $\Lambda$-adic 
Eisenstein series on $U(2,1)_{/\cF}$ whose constant term is a product of $L_{\Sigma,\lam}$ and a Deligne-Ribet $p$-adic L-function $L_{DR}$, which contribute to two distinct Selmer groups respectively. 
It turns out that the Eisenstein congruence method does not directly allow to distinguish their contribution to the individual Selmer groups, unless the common zeros of $L_{\Sigma,\lam}$ and $L_{DR}$ 
are simple zeros. When the global root number of $\lam$ equals one, the results on $\mu(L_{\Sigma,\lam}^{-})$ imply that $L_{\Sigma,\lam}$ and $L_{DR}$ have no common zeros. In the case of root number $-1$, 
Theorem A implies that $L_{\Sigma,\lam}$ and $L_{DR}$ have at most one simple zero. Thus, the proof of the one-sided divisibility can be completed.\\
\\
When $\lam$ is self-dual without any condition on the root number, the method in this paper can be used to 
study the non-triviality modulo $p$ of certain higher order anticyclotomic derivatives $L^{(k)}_{\Sigma,\lam}$. 
As a consequence, we can show that $L_{\Sigma,\lam}\in \Zbarp\powerseries{\Gamma}$ is not a polynomial. It would be interesting to see whether this fact has 
an arithmetic consequence.\\
\\
It seems likely that $L_{\Sigma,\lam}^{'}$ generates the characteristic ideal of the torsion part of the anticyclotomic Selmer group 
upto an anticyclotomic regulator.
Another interesting question would be whether a normalisation of $L_{\Sigma,\lam}^{'}$ interpolates a 
normalisation of (complex) derivative L-values.\\
\\
The article is perhaps a follow up to the articles \cite{Hi} and \cite{Hs1}. We refer to them for general introduction. 
In exposition, we suppose that the reader is familiar with them, particularly with \cite{Hs1}.\\
\\
The article is organised as follows. In \S 2, we obtain an expression for the power series associated to the cyclotomic derivative of Katz $p$-adic L-function
in terms of the $t$-expansion of certain $p$-adic Hilbert modular forms around a well chosen CM point (cf. (1.9)). In \S2.1, we firstly recall the formula for the Fourier coefficients of
the toric Eisenstein series in \cite{Hs1} where it is used to construct Katz $p$-adic L-function. In \S2.2, we then construct the $p$-adic modular forms  ${\bf{f}}_{\lam',i}$ 
and obtain the expression (1.9). 
In \S3, we prove Theorem A. In \S3.1, we firstly prove the vanishing of the toric Eisenstein series associated to $\lam$. 
In \S3.2-3.4, Theorem A is then proven. 
In \S4, as an application, we prove the non-vanishing of the anticyclotomic regulator in \cite{A}.\\
\begin{thank}
We are grateful to our advisor Haruzo Hida for continuous guidance and encouragement. The question was suggested by him.
We thank Ming-Lun Hsieh for patiently answering our questions regarding \cite{Hs1} and \cite{Hs2} and also for his assistance. 
We thank friends for prompting us to stay in UCLA during the summer of 2011. Most of the current work was done during this period. 
We thank Adebisi Agboola, Mahesh Kakde, Chandrashekhar Khare, Richard Taylor and Kevin Ventullo for interesting conversations about the topic. 
Finally, we thank the referees for constructive criticism and helpful suggestions.\\
\end{thank}
\noindent {\bf{Notation}}
We use the following notation unless otherwise stated.\\
\\
Let $\cF_+$ denote the totally positive elemnts in $\cF$. We sometime use $O$ to denote the ring of integers $\cO_\cF$. Let $\cD_\cF$ (resp. $D_\cF$) be the 
different (resp. discriminant) of $\cF/\Q$.  Let $\cD_{\cK/\cF}$ (resp. $D_{\cK/\cF}$) be the 
different (resp. discriminant) of $\cK/\cF$. Let ${\bf{h}}$ (resp. ${\bf{h}}_\cK$) be the set of finite places of $\cF$ (resp. $\cK$).
Let $v$ be a place of $\cF$ and $w$ be a place of $\cK$ above $v$. Let $\cF_v$ be the completion of $\cF$ at $v$, $\varpi_v$ an uniformizer 
and $\cK_{v} = \cF_{v} \otimes_{\cF} \cK$. For the $p$-ordinary CM type $\Sg$ of $\cK$ as above, let
$\Sg_p = \{ w \in {\bf{h}}_\cK | w|p$ and $w$ induced by $\iota_p \circ \sigma$, for $\sigma \in \Sg\}$.\\
\\
For a number field $\cL$, let ${\bf{A}}_{\cL}$ be the adele ring, ${\bf{A}}_{\cL,f}$ the finite adeles and ${\bf{A}}_{\cL,f}^{\Box}$ the finite adeles away from a finite set of places 
$\Box$ of $\cL$. For $a \in \cL$, let $\mathfrak{il}_{\cL}(a)=a(\cO_\cL \otimes_{\Z} \widehat{\Z}) \cap \cL$. 
Let $G_\cL$ be the absolute Galois group of $L$ and $\rec_{\cL}: {\bf{A}}_{\cL}^\times \rightarrow G_{\cL}^{ab}$ the geometrically normalized reciprocity law. 
Let $\psi_\Q$ be the standard additive character of ${\bf{A}}_\Q$
such that $\psi_\Q(x_\infty)=\exp(2\pi i x_{\infty})$, for $x_\infty \in \R$. Let $\psi_\cL : {\bf{A}}_{\cL}/ \cL \rightarrow \C$ be given by 
$\psi_{\cL}(y)=\psi_{\Q}\circ (\Tr_{\cL/\Q}(y))$, for $y \in {\bf{A}}_{\cL}$. We denote $\psi_\cF$ by $\psi$.\\
\\
\section{Cyclotomic derivative}
\def\ads{\chi}
\def\nads{\chi^*}
\def\addchar{\psi}
\def\Fv{F}
\def\Kv{E}
\def\OFv{O_v}
\def\OF{O}
\noindent In this section, we obtain an expression for the power series associated to the cyclotomic derivative of Katz $p$-adic L-function
in terms of the $t$-expansion of certain $p$-adic Hilbert modular forms around a well chosen CM point (cf. (1.9)).\\
\\
Let $\ads$ be a Hecke character of infinity type $k\Sg + \kappa(1-c)$, where $k\geq 1$ and 
$\kappa = \sum {\kappa_{\sigma} \sigma} \in \Z[\Sigma]$ with $\kappa_{\sigma} \geq 0$. In \S 2.1, 
we recall the formula for the Fourier coefficients of
the toric Eisenstein series in \cite{Hs1} where it is used to construct Katz $p$-adic L-function for $\chi$. 
In \S 2.2, we obtain an expression for the power series associated to the cyclotomic derivative of Katz $p$-adic L-function
in terms of the $t$-expansion of certain $p$-adic Hilbert modular forms around a well chosen CM point (cf. (1.9)).\\
\\
\subsection{Toric Eisenstein series}
In this subsection, we recall the formula for the Fourier coefficients of
the toric Eisenstein series in \cite{Hs1} where it is used to construct Katz $p$-adic L-function.\\
\\
Suppose that $\frakC$ is the prime-to-$p$ conductor of $\ads$. 
We write $\frakC=\frakC^+\frakC^-$ such that $\frakC^+$ (resp. $\frakC^-$) is a product of prime factors split (resp. non-split) over $\cF$. 
We further decompose $\frakC^+=\Csplit\Csplit_c$ such that $(\Csplit,\Csplit_c)=1$ and $\Csplit\subset\Csplit_c^c$.
Let 
\[\frakD=p\frakC\frakC^c D_{\cK/\cF}.\]\\
Let $\cU_p$ be the torsion subgroup of $\cO_{\cF_p}^\x$. Let $u=(u_v)_{v|p}\in\cU_p$. Let $\bfc=(\bfc_v)\in \A_{\cF,f}^{\times}$ such that $\bdc_v=1$ at $v|\frakD$ and 
$\frakc=\mathfrak{il}_{\cF}(\bfc)$.\\
\\
We put
 \[\nads=\ads\Abs^{-\onehalf}_{\AK}\text{ and }\ads_+=\ads|_{\AF^\x}.\]
Let $v$ be a place of $\cF$. Let $\Fv=\cF_v$ (resp. $\Kv=\cK\ot_{\cF}\cF_v$) and $O_v$ be the ring of integers of $\Fv$. 
Denote by $z\mapsto\zbar$ the complex conjugation. Let $\Abs$ be the standard absolute values on $F$ and 
let $\Abs_E$ be the absolute value on $E$ given by $\abs{z}_E:=\abs{z\zbar}$. Let $d_F=d_{\cF_v}$ be a 
fixed generator of the different $\cD_\cF$ of $\cF/\Q$. 
For a set $Y$, denote by $\bbI_Y$ the characteristic function of $Y$. Let $\theta_v \in E$ be as in \cite[\S3.2]{Hs1}.\\
\begin{prop}\label{P:1.N} There exists a toric Hilbert modular Eisenstein series  $\holES_{\ads,u}$ of weight $k$ with the following property. 
Say the $q$-expansion of $\holES_{\ads,u}$ at the cusp $(O,\frakc^{-1})$ is given by
\[\holES_{\ads,u}|_{(\OF,\frakc^{-1})}(q)=\sum_{\beta\in \cF_+ }\bfa_\beta(\holES_{\ads,u},\frakc)\cdot q^\beta.\]
Then, the $\beta$-th Fourier coefficient $\bfa_\beta(\holES_{\ads,u},\frakc)$ is given by
\begin{align*}\bfa_\beta(\holES_{\ads,u},\frakc)=&\beta^{(k-1)\Sg}\prod_{w|\Csplit}\chi_{w}(\beta)\bbI_{\OFv^\x}(\beta)\prod_{w\in\Sg_p}\chi_w(\beta)\bbI_{u_v(1+\uf_v\OFv)}(\beta)\\
 &\times \prod_{v\ndivide \infty\frakD}\bfa_{\beta,v}(\bfc_v,\chi) \prod_{v|\frakC^- D_{\cK/\cF}}A_{\beta,v}(\chi),
\end{align*}
where $\bfa_{\beta,v}(\bfc_v,\chi)=\sum_{i=0}^{v(\bfc_v)}\nads(\uf_v^i)$, $A_{\beta,v}(\chi)= L(0,\ads_v)\wtd A_{\beta}(\ads_v)$ and
\beq\label{E:AB.V}\begin{aligned} \wtd A_\beta(\ads_v)=&\int_{\cF_v}\ads_v^{-1}\Abs_E^s(x_v+\bftheta_v)\addchar(-d_{\cF_v}^{-1}\beta x_v)dx_v|_{s=0}\\
:=&\lim_{n\to\infty}\int_{\uf_v^{-n}\cO_{\cF_v}}\ads_v^{-1}(x_v+\bftheta_v)\addchar(-d_{\cF_v}^{-1}\beta x_v)dx_v.\end{aligned}
\eeq
\end{prop}
\begin{proof} The construction is given in \cite[\S4.1-4.5]{Hs1}. The formula follows from the calculations of local Whittaker integrals of the special local 
sections in \cite[\S 4.3]{Hs2} (\cf \cite[Prop.\,4.1 and Prop.\,4.4]{Hs1}).\\
\end{proof}
\noindent
\begin{remark} Note that the formulas for the local Whittaker integrals 
$\bfa_{\beta,v}({\bf{c}}_v,\chi)$ and $A_{\beta,v}(\chi)$ are well defined for $\beta \in \cF_v$ as well. This is used later in \S3.3 and \S3.4.\\
\end{remark}
\noindent For $a\in \A_{\cK,f}^{\setp}$, let $\mathfrak{c}(a)=\mathfrak{c}(\cO_\cK)N_{\cK/ \cF}(\mathfrak{a})$. 
Here $\mathfrak{a}=\mathfrak{il}_{\cF}(a)$ and $\mathfrak{c}(\cO_\cK)$ is as in \cite[\S3.4]{Hs1}.\\
\\
Let $\holES_{\ads,u,a}=\holES_{\ads,u}|_{\mathfrak{c}(a)}$ (cf. \cite[Prop. 4.4]{Hs1}). Let $\mathcal{E}_{\chi,u}$ (resp. $\mathcal{E}_{\chi,u,a}$) be the $p$-adic avatar of $\holES_{\ads,u}$ (resp. $\holES_{\ads,u,a})$
i.e. the corresponding $p$-adic Hilbert modular form. Let $\widehat{\chi}$ be the $p$-adic avatar of $\chi$ i.e. the corresponding $p$-adic Hecke character. Let 
$\mathcal{E}_{\widehat{\chi},u}=\theta^{\kappa} \mathcal{E}_{\ads,u}$ and $\mathcal{E}_{\widehat{\chi},u,a}=\theta^{\kappa} \mathcal{E}_{\ads,u,a}$, 
where $\theta^{\kappa}= \prod_{\sigma} \theta(\sigma)^{\kappa_{\sigma}}$ and $\theta(\sigma)$ 
is the Katz $p$-adic differential operator corresponding to $\sigma$ (cf. \cite[cor. 2.6.25]{Ka}). For later purpose, we recall the following fact.\\
\\
Say the $q$-expansion of $\mathcal{E}_{\widehat{\chi},u}$ at the cusp $(O,\frakc^{-1})$ is given by
$\mathcal{E}_{\widehat{\chi},u}|_{(\OF,\frakc^{-1})}(q)=\sum_{\beta\in \cF_+ }\bfa_\beta(\mathcal{E}_{\widehat{\chi},u},\frakc)\cdot q^\beta.$
Then, the $\beta$-th Fourier coefficient $\bfa_\beta(\mathcal{E}_{\widehat{\chi},u},\frakc)$ is given by
\beq
\bfa_\beta(\mathcal{E}_{\widehat{\chi},u},\frakc)=\iota_p(\beta^{\kappa} \bfa_\beta(\holES_{\ads,u},\frakc)).
\eeq
This follows from the effect of the Katz $p$-adic differential operator on the $q$-expansion of $p$-adic Hilbert modular forms (cf. \cite[cor. 2.6.25]{Ka}). 
In what follows, we drop $\iota_p$ from the notation.\\
\noindent \\
\subsection{Cyclotomic derivative}
In this subsection, we obtain an expression for the power series associated to the cyclotomic derivative of Katz $p$-adic L-function
in terms of the $t$-expansion of certain $p$-adic Hilbert modular forms around a well chosen CM point (cf. (1.9)).\\
\\
Let $Cl_-:=\cK^\x\AFf^\x\bksl \AKf/U_\cK$ and let $Cl_-^{alg}$ be the subgroup of $Cl_-$ generated by ramified primes. 
Let $\Gamma'$ be the open subgroup of $\Gamma^-$ generated by the image of $O_{p}^\times \x\prod_{v|D_{\cK/\cF}}\cK_v^\x$ via $\rec_\cK$. 
The reciprocity law $\rec_\cK$ at $\Sg_p$ induces an injective map $\rec_{\Sg_p}\colon 1+pO_{p}\hookto O_{p}^\x=\oplus_{w\in\Sg_p}\cO_{\cK_w}^\x\stackrel{\rec_\cK}\longto Z(\frakC)^-$ 
with finite cokernel as $p\ndivides D_{\cF}$, and it is easy to see that $\rec_{\Sg_p}$ induces an isomorphism $\rec_{\Sg_p}:1+pO_p\isoto\Gamma'$. 
We thus identify $\Gamma'$ with the subgroup $\rec_{\Sg_p}(1+pO_p)$ of $Z(\frakC)^-$. Let $Z':=\pi_-^{-1}(\Gamma')$ be the subgroup of 
$Z(\frakC)$ and let $Cl'_-\supset Cl^{alg}_-$ be the image of $Z'$ in $Cl_-$ and let $\cD'_1$ (resp. $\cD_1''$) be a set of representatives of $Cl'_-/Cl^{alg}_-$ (resp. $Cl_-/Cl'_-$) in 
$(\AKf^{(\frakD)})^\x$. Let $\cD_1:=\cD_1''\cD'_1$ be a set of representatives of $Cl_-/Cl^{alg}_-$. Let $\cU_p$ be the torsion subgroup of $(\cO_{\cF}\ot_\Z\Zp)^\x$ and 
let $\cU^{alg}:=U_\cK\cap (\cK^\x)^{1-c}$. Let $\cD_0$ be a set of representatives of  $\cU_p/\cU^{alg}$ in $\cU_p$.\\
\\
Suppose that $\lam$ is a self-dual Hecke character of infinity type $k\Sigma$ with the root number $-1$, where $k>0$.\\
\\
Let $\mathcal{E}_{\lam',u}$ (resp. $\mathcal{E}_{\lam',u,a}$) be the $p$-adic derivative of $\mathcal{E}_{\widehat{\lam N^{p^n}},u}$ (resp. $\mathcal{E}_{\widehat{\lam N^{p^n}},u,a}$)
 in the space of $p$-adic Hilbert modular forms i.e. 
\beq\label{E}\mathcal{E}_{\lam',u} = \lim_{n \rightarrow \infty} \frac { \mathcal{E}_{\widehat{\lam N^{p^{n}}},u} - \mathcal{E}_{\widehat{\lam},u}} {p^{n}}\eeq
and $\mathcal{E}_{\lam',u,a}$ defined analogously. From the formula for Fourier coefficients in Proposition 2.1 and (2.2), 
we first note that the corresponding $q$-expansion limit exists. The existence of $\mathcal{E}_{\lam',u}$ and $\mathcal{E}_{\lam',u,a}$
in the space of $p$-adic Hilbert modular forms then follows from the fact that $p$-adic limit of $p$-adic Hilbert modular forms is again a $p$-adic Hilbert modular form (cf. \cite[\S1.9]{Ka}).\\
\\
Say the $q$-expansion of $\mathcal{E}_{\lam',u}$ at the cusp $(O,\frakc^{-1})$ is given by
$\mathcal{E}_{\lam',u}|_{(\OF,\frakc^{-1})}(q)=\sum_{\beta\in \cF_+ }\bfa_\beta(\mathcal{E}_{\lam',u},\frakc)\cdot q^\beta$.
Let $\bfa_{\beta,v}'(\bfc_v,\lam)$ and $A_{\beta,v}'(\lam)$ be the $p$-adic derivatives of $\bfa_{\beta,v}(\bfc_v,\lam N^{p^{n}})$ and  $A_{\beta,v}(\lam N^{p^{n}})$, respectively.\\
\\
The following theorem gives a formula for $\mu(L_{\Sigma,\lam}^{'})$ in terms of the $p$-adic valuation of the Fourier coefficients $\bfa_\beta(\mathcal{E}_{\lam',u},\frakc)$.\\
\begin{thm} Suppose that $p\ndivide h_\cK^-\cdot D_\cF$. Then, we have
\[\mu(L_{\Sigma,\lam}^{'})=\inf_{(u,a)\in \cD_0\x\cD_1, \beta \in \cF_+}v_p( \frac{\bfa_\beta(\mathcal{E}_{\lam',u},\frakc(a))}{\log_{p}(1+p)}).\]
\end{thm}
\begin{proof}
Let $\bfx$ be the CM point with CM type $(\cK,\Sg)$ and the polarization ideal $\frakc(\cO_\cK)$ defined in 
\cite[\S 5.2]{Hs1}. Let $t$ denote the Serre-Tate co-ordinates of the deformation space of $\bfx$. 
For $a\in \cD_1'$, let $\langle a \rangle_{\Sigma}$ be the unique element in $1+pO_p$ such that 
$\rec_{\Sg_p}(\langle a \rangle_{\Sigma})=\pi_-(\rec_\cK(a))\in\Gamma'$. We define a formal $t$-expansion $\cE_{\widehat{\lam N^k}}(t)$ by 
\[\cE_{\widehat{\lam N^k}}(t):=\sharp(\mathcal{U}^{alg})\cdot\sum_{(u,a)\in \cD_0\x \cD_1' }\ads(a)\mathcal{E}_{\widehat{\lam N^k}, u,a}|[a] (t^{\langle a \rangle_{\Sigma}u^{-1}}),\]
where $|[a]$ is the Hecke action induced by $a$ (cf. \cite[Remark 4.5]{Hs1}).\\
\\
By an argument identical to the proof of \cite[Prop. 5.2]{Hs1}, it follows that the formal $t$-expansion $\cE_{\widehat{\lam N^k}}(t)$ equals the power series expansion of the measure 
$L_{\Sigma,\lam N^k}^{-}$ regarded as a $p$-adic measure on $O_p$ with support in $1+pO_p$ i.e. 
$G_{\Sigma,\lam N^k}=\cE_{\widehat{\lam N^k}}(t)$, where $G_{\Sigma,\lam N^k}$ is as in (1.7).\\
\\
In view of the introduction, it now follows that the formal $t$-expansion  $\cE_{\lam'}(t)$ defined by
\[\cE_{\lam'}(t):=\frac{\sharp(\mathcal{U}^{alg})}{\log_p(1+p)}\cdot\sum_{(u,a)\in \cD_0\x \cD_1' }\ads(a)\mathcal{E}_{\lam', u,a}|[a] (t^{\langle a \rangle_{\Sigma}u^{-1}}),\]
equals the power series expansion of $L_{\Sigma,\lam}'$. We thus conclude that
$$\mu(L_{\Sigma,\lam}')=\inf\stt{r\in\Q_{\geq 0}\mid p^{-r}\cE_{\lam'}(t)\not\con 0\pmod{\frakm_{\Zbarp}}}.$$
\\
Here $\frakm_{\Zbarp}$ is the maximal idea of the local ring $\Zbarp$. 
Note that $ p \ndivide \sharp(\mathcal{U}^{alg})$. From \cite[Lemma 5.3]{Hs1}, the linear independence of mod $p$ Hilbert modular forms (cf. \cite[Thm. 3.20]{Hi}), 
and the $q$-expansion principle of \padic modular forms, it follows that
\[\mu(L_{\Sigma,\lam}^{'})=\inf_{(u,a)\in \cD_0\x\cD_1, \beta \in \cF_+}v_p( \frac{\bfa_\beta(\mathcal{E}_{\lam',u},\frakc(a))}{\log_{p}(1+p)}).\]
\end{proof}
\noindent\\
\section{Proof of Theorem A}
\noindent In this section, we prove Theorem A. 
In \S 3.1,  we firstly prove the vanishing of the toric Eisenstein series associated to a self-dual Hecke character of root number $-1$ in \S2.1. In \S3.2-3.4, we prove the Theorem.\\
\\
\subsection{Vanishing of toric Eisenstein series}
In this subsection,  we prove the vanishing of the toric Eisenstein series associated to a self-dual Hecke character of root number $-1$ in \S2.1.\\
\\
For $\delta$ as in \cite[$(d1)$ and $(d2)$]{Hi}, let $\xi=2\delta$ . 
We recall the formula for the local root number of a self-dual Hecke character.\\
\begin{lm}
Let $\lam$ be a self-dual Hecke character, $v$ a place of $\cF$ and  $W(\lam_{v}^{*})$ the corresponding local root number. Then,
$$ W(\lam_{v}^{*})=\pm\lam_{v}^{*}(\xi).$$
Moreover,\\
(1). If  $v=\sigma \in \Sigma$, then $W(\lambda_{\sigma}^{*}) = \lambda_{\sigma}^{*}(\xi)$,\\
(2). If $v$ is split, then $W(\lam_{v}^{*})=\lam_{v}^{*}(\xi)$ and\\
(3). If $v$ is non-split finite, then $W(\lam_{v}^{*})=(-1)^{a(\lam_{v}^{*})+v(\mathfrak{c}(\cO_\cK))}\lam_{v}^{*}(\xi)$,
where $\mathfrak{c}(\cO_\cK)=\cD_{\cF}^{-1}(\xi \cD_{\cK/\cF}^{-1})$ and $a(\lam_{v}^{*})=\inf\big{\{}n \in \Z_{\geq0}|\lam_{v}^{*}|_{1+\varpi_{v}^{n}O_{v}}=1 \big{\}}$ (cf. \cite[Prop. 3.7]{MS} and \cite{Ta}).\\
\end{lm}
\noindent Now suppose that $\lam$ is a self-dual Hecke character of infinity type $k\Sigma$ with the root number $-1$, where $k>0$.\\
\begin{prop} The toric Eisenstein series $\mathcal{E}_{\lam,u}$ vanishes.
 Moreover, for a given $\beta \in \cF_+$ coprime-to-$\mathfrak{F}$ such that $\beta_v \in u_{v}(1+\varpi_{v}O_v)$ for all $v|p$,
 there exists a non-split place $v_1$ such that
 $\bfa_{\beta,v_1}(\bfc_{v_1},\lam)=0$.
\end{prop}
\begin{proof}
 Let $\beta \in \cF_+$. The global root number $W(\lam^{*})$ is given by  $W(\lam^{*})=\prod_{v} W(\lam_{v}^{*}) = -1$. 
Note that $ W(\lam^{*})\tau_{\cK/\cF}(\beta)=(-1)\cdot\lam^{*}(\xi).$
In view of the fact that $\beta \in \cF_+$ and Lemma 3.1, it thus follows that there exists a place $v$ such that\\
$$W(\lam_{v}^{*})\tau_{\cK_v/\cF_v}(\beta)=(-1)\cdot\lam_{v}^{*}(\xi).$$\\
Here $\tau_{\cK_{v}/\cF_{v}}$ denotes the
character associated to the extension $\cK_{v}/\cF_{v}$. In view of Lemma 3.1, we conclude that this $v$ is finite and non-split. Now there are two cases.\\
\\
Case I - $v |\cond(\lam)$. Here $\cond(\lam)$ is the conductor of $\lam$. From the above equation and \cite[Lem. 6.1]{Hs1}, it now follows that $A_{\beta,v}(\lam)=0$.\\
\\
Case II - $v \ndivide\cond(\lam)$. From Lemma 3.1, it follows that this $v$ is inert and the valuation $v(\bfc_{v}\beta)$ is odd.  From the 
formula in Proposition 2.1, we conclude that $\bfa_{\beta,v}(\bfc_{v},\lam)=0$.\\
\\
The vanishing of the toric Eisenstein series $\mathcal{E}_{\lam,u}$ thus follows from the decomposition of the $\beta$-th Fourier coefficient as a product of local Whittaker integrals in Proposition 2.1.\\
\end{proof}
\noindent 
\begin{remark} The vanishing also follows from the linear independence of mod $p$ modular forms (cf. \cite[Thm. 3.20]{Hi}). 
The above more elementary proof is due to the referee.
\end{remark}
\noindent Here is a useful corollary.\\
\begin{cor} Let $\beta \in \cF_{+}$. A necessary condition for the non-vanishing of  $\bfa_{\beta}(\mathcal{E}_{\lambda^{'},u}, \mathfrak{c})$ is that exactly one of the local Whittaker coefficients
 $\bfa_{\beta,v}(\bfc_v,\lam)$ and $A_{\beta,v}(\lam)$ vanish. 
 Moreover, such a $v$ must be non-split. 
 \end{cor}
\begin{proof} The first part of the corollary follows by Proposition 3.1, the Leibnitz product rule and the Fourier coefficient formulas in Proposition 2.1 and (2.2).\\
\\
If $\beta$ is as in Proposition 3.2, then the second part follows immediately by the lemma. If it not of this
form, then by the formulas in Proposition 2.1 and (2.2)
$$\bfa_{\beta}(\mathcal{E}_{\widehat{\lambda N^{k}},u}, \frakc)=0$$ 
for any $k$. In particular, $\bfa_{\beta}(\mathcal{E}_{\lambda^{'},u}, \mathfrak{c})=0$.\\
\end{proof}
\noindent \\
\subsection{Lower bound} In this subsection, 
we prove the lower bound 
$$\mu(L_{\Sigma,\lam}^{'}) \geq \min_{v|\mathfrak{C^{-}}} \{\mu_{p}'(\lam), \mu_{p,v}'(\lam)\}$$
of the equality asserted in Theorem A.\\
\\
Let $v$ be a local place of $\cF$ and $|\cdot|=|\cdot|_v$ be the corresponding absolute value.\\
\\
Let $\lam$ be as before, $N:=|\cdot|_{\A_{\cK}^\times}$ be the norm Hecke character and $\lam^{*}=\lam \cdot N^{-1}$. From the self-duality, 
we see 
\beq
\label{E}\lambda^{*}|_{\cF_{v}^{\times}}= \tau_{\cK_{v}/\cF_{v}}.
\eeq 
Here $\tau_{\cK_{v}/\cF_{v}}$ denotes the
character associated to the extension $\cK_{v}/\cF_{v}$.\\
\\
Let $\lam'$ be the $p$-adic derivative of $\lam N^{p^k}$.
Recall that $\bfa_{\beta,v}'(\bfc_v,\lam)$ and $A_{\beta,v}'(\lam)$ are the $p$-adic derivatives of $\bfa_{\beta,v}(\bfc_v,\lam N^{p^{n}})$ and  $A_{\beta,v}(\lam N^{p^{n}})$, respectively.\\
\begin{lm} Let $c=1$ if $v$ is non-split and $c=2$, otherwise.\\
(1). If $u \in \cO_{\cF_v}^{\times}$, then $\lambda'(u\varpi_{v}^{n})= cn\tau_{K/F}(u)(-|\varpi_{v}|)^{n}\log_{p}(|\varpi_{v}|)$.\\
(2). $\lambda'^{*}(u\varpi_{v}^{n})= (-1)^{n}cn\tau_{K/F}(u)\log_{p}(|\varpi_{v}|).$ \\     
(3). If $v\ndivide p$ and $\beta \in \cF_{v}^{\times}$, then $\log_{p}(|\varpi_{v}|)$ divides $\bfa_{\beta,v}'(\bfc_v,\lam)$ and $A_{\beta,v}'(\lam)$.
\end{lm}
\begin{proof}
From (3.1) it follows that
\beq
\label{E}\lambda(u\varpi_{v}^{n})= \tau_{\cK/\cF}(u\varpi_{v}^{n})|\varpi_{v}|^{n}= \tau_{\cK/\cF}(u) (-|\varpi_{v}|)^{n}.
\eeq
Now $N = |\cdot|_{\A_{\cK}^\times}$. Thus, $N^{k}(u\varpi_{v}^{n})= |\varpi_{v}|^{cnk}$. 
Fixing $u$ and $n$, we can think $N^{k}(u\varpi_{v}^{n})$
as a function of $k$. 
Note that 
\beq
\label{E}\lambda'(u\varpi_{v}^{n})= \tau_{\cK/\cF}(u)(-|\varpi_{v}|^{n})\log_{p}(|\varpi_{v}|^{cn})=
cn\tau_{K/F}(u)(-|\varpi_{v}|)^{n}\log_{p}(|\varpi_{v}|).
\eeq
In particular, $\lambda'(u)=0$. In any case, $\log_{p}(|\varpi_{v}|)$ divides $\lambda'(u\varpi_{v}^{n})$.
Also note that 
\beq
\label{E}\lambda'^{*}(u\varpi_{v}^{n})= |\varpi_{v}|^{-n} \lambda'(u\varpi_{v}^{n})= (-1)^{n}cn\tau_{K/F}(u)\log_{p}(|\varpi_{v}|).
\eeq
It now follows from the formulas in \cite[\S4.3]{Hs2} that $\log_{p}(|\varpi_{v}|)$ 
divides $\bfa_{\beta,v}'(\bfc_v,\lam)$ and $A_{\beta,v}'(\lam)$. 
\\
\end{proof}
\noindent We are now ready to prove the lower bound.\\
\begin{prop}  $$\mu(L_{\Sigma,\lam}^{'}) \geq \min_{v|\mathfrak{C^{-}}} \{\mu_{p}^{'}(\lam), \mu_{p,v}'(\lam)\}.$$
\end{prop}
\begin{proof}By Theorem 2.2, it suffices to show that for every $(u,a) \in \cD_{0} \times \cD_{1}$ and $\beta \in \cF_{+}$, 
we have
\beq
\label{E}
v_{p}(\frac{ \bfa_{\beta}(\mathcal{E}_{\lambda',u}, \mathfrak{c}(a))}{\log_{p}(1+p)} )\geq  \min_{v|\mathfrak{C^{-}}} \{\mu_{p}^{'}(\lam), \mu_{p,v}'(\lam)\}.
\eeq
In view of the formula for the Fourier coefficients, we only need to consider that the case $\beta$ is coprime-to-$\mathfrak{F}$ and 
$\beta \in u_{v}(1+\varpi_{v}O_v)$ for all $v|p$.
We can also suppose that there exists exactly one non-split place $v$ that exactly one of the local Whittaker coefficients
 $\bfa_{\beta,v}(\bfc_v,\lam)$ and $A_{\beta,v}(\lam)$ vanishes (cf. Corollary 3.3).\\
\\
We have the following two cases.\\
\\
Case I - $v_{1}|\mathfrak{C^{-}}$. 
From the formula for the Fourier coefficients, the Leibnitz product rule, it follows that\\
\beq
 \bfa_{\beta}(\mathcal{E}_{\lambda',u}, \mathfrak{c}(a))=\beta^{(k-1)\Sg}\prod_{w|\Csplit}\lam_{w}(\beta)\prod_{w\in\Sg_p}\lam_w(\beta)
 \prod_{v\ndivide \infty\frakD}\bfa_{\beta,v}(\bfc_v,\lam) \cdot A_{\beta,v_1}'(\lam)\prod_{v\neq v_1|\frakC^- D_{\cK/\cF}}A_{\beta,v}(\lam).
\eeq
\noindent\\
By part (3) of Lemma 3.4, 
$$v_{p} ( \frac{ A_{\beta,v_1}'(\lam)}{\log_{p}(1+p)}) \geq v_{p}(\frac{\log_{p}(|\varpi_{v_{1}}|)}{\log_{p}(1+p)}).$$
If $v|\mathfrak{C}^-$, then in \cite[(4.16) and (4.17)]{Hs2} it is shown that
$$v_{p}(A_{\beta,v}(\lam) ) \ge \mu_{p}(\lambda_{v}).$$\\
Thus,
\beq
\label{E}
v_{p}(\frac{ \bfa_{\beta}(\mathcal{E}_{\lambda',u}, \mathfrak{c}(a))}{\log_{p}(1+p)} )\geq  \mu_{p,v_1}'(\lam).
\eeq
\\
Case II - $v_{1} \ndivide \mathfrak{C^{-}}$. From the formula for the Fourier coefficients and the Leibnitz product rule, it follows that\\
\beq
 \bfa_{\beta}(\mathcal{E}_{\lambda',u}, \mathfrak{c}(a))=\beta^{(k-1)\Sg}\prod_{w|\Csplit}\lam_{w}(\beta)\prod_{w\in\Sg_p}\lam_w(\beta)
 \cdot \bfa_{\beta,v_1}'(\bfc_{v_1},\lam)\prod_{v\neq v_1, v\ndivide \infty\frakD}\bfa_{\beta,v}(\bfc_v,\lam) \prod_{v|\frakC^- D_{\cK/\cF}}A_{\beta,v}(\lam).
\eeq
\noindent\\
By a similar argument as in the previous case, we conclude that
\beq
\label{E}
v_{p}(\frac{ \bfa_{\beta}(\mathcal{E}_{\lambda',u}, \mathfrak{c}(a))}{\log_{p}(1+p)} )\geq  \mu_{p}'(\lam).
\eeq
We have thus proven (3.5) and this finishes the proof.\\
\end{proof}
\noindent \\
\subsection{Upper bound I }
In this subsection, we prove an upper bound 
$$\mu(L_{\Sigma, \lambda}^{'}) \leq \mu_{p}'(\lam)$$
of the equality asserted in Theorem A.\\
\\
Let ${\bf{c}}^{0}\in \A_{\cF}$ such that $\mathfrak{il}_{\cF}({\bf{c}}^{0}) = \mathfrak{c}(\cO_\cK)$. We recall that $\mathfrak{c}(\cO_\cK)$ is as in\cite[\S3.1]{Hs1}.\\
\\
We begin with a couple of local lemmas.\\
\begin{lm} There exist infinitely many inert primes $\mathfrak{q}_1$ and $\eta_{\mathfrak{q}_{1}} \in \cF_{\mathfrak{q}_{1}}^{\times}$ such that the following conditions hold:\\
(1). $$v_{p}(\frac{ \bfa_{\eta_{\mathfrak{q}_1}}'(\lam, {\bf{c}}_{\mathfrak{q}_1}^{0})}{\log_{p}(1+p)})=0,$$
(2). $$ W(\lam_{\mathfrak{q}_1}^{*})\tau_{\cK_v/\cF_v}(\eta_{\mathfrak{q}_1})=(-1)\cdot\lam_{\mathfrak{q}_1}^{*}(\xi).$$
\end{lm}
\begin{proof}
We first consider the condition
\beq
\label{E}v_{p}(\frac {\log_{p}(|\varpi_{\mathfrak{q}_{1}}|)}{\log_{p}(1+p)})=0.
\eeq
The condition amounts to $q_1 \equiv a (\mod{p})$ but $q_1 \not\equiv a (\mod{p^2})$ and $ p \ndivide [\cF_{\mathfrak{q}_1}:\Q_{q_1}]$, where $(q_1)=\mathfrak{q}_1$,
$q_{1} > 0$ and $1\leq a \leq p-1$. 
From Dirichlet density theorem and Chebotarev density theorem, it now follows that there exist  infinitely many inert primes $\mathfrak{q}_1$ satisfying (3.10). 
We choose any such $\mathfrak{q}_1 \ndivide \mathfrak{D}$.\\
\\
From the local Whittaker formula (cf. Prop. 3.1 and (3.20) and the Leibnitz product rule, it now follows that  
\beq
\label{E}\bfa_{\eta_{\mathfrak{q}_1}}'(\lam, {\bf{c}}_{\mathfrak{q}_1}^{0}) = 
\log_{p}(|\varpi_{\mathfrak{q}_{1}}|)\sum_{i=0}^{v_{\mathfrak{q}_{1}}(\eta_{\mathfrak{q}_{1}}  {\bf{c}}_{\mathfrak{q}_1}^{0})} (-1)^{i} i .
\eeq
\\
Note that $\sum_{i=0}^{i=n} (-1)^{i}i$ equals $\frac{n}{2}$ if $n$ is even and $\frac{-(n+1)}{2}$ if $n$ is odd.
The existence of $\eta_{\mathfrak{q}_{1}} \in \cF_{\mathfrak{q}_{1}}^{\times}$ satisfying (1) thus follows.\\
\\
From part (3) of Lemma 3.1, (2) holds whenever $v_{\mathfrak{q}_{1}}(\eta_{\mathfrak{q}_{1}}) \equiv v_{\mathfrak{q}_{1}}({\bf{c}}_{\mathfrak{q}_1}^{0})+ 1 (\mod{2})$. 
Thus, there exists  $\eta_{\mathfrak{q}_{1}} \in \cF_{\mathfrak{q}_{1}}^{\times}$ satisfying (1) and (2).\\
\end{proof}
\noindent For convenience, let us state \cite[Prop. $6.3$]{Hs1} as the following lemma.\\
\begin{lm}(Hsieh) Let $v|\mathfrak{C^{-}}$. There exists an $\eta_{v} \in \cF_{v}^{\times}$ such that
the following conditions hold:\\
(1).$$ v_{p}( A_{\eta_{v},v}(\lam) ) = \mu_{p}(\lam_{v}), $$
(2). $$ W(\lambda_{v}^{*})\tau_{\cK_{v}/\cF_{v}}(\eta_{v}) = \lambda_{v}^{*}(\xi).$$
\end{lm}
\noindent \\
With enough preparations, we have the following proposition.\\
\begin{prop} There exists $\beta \in \cF_{+}$, $u \in \cD_0$ and $\mathfrak{c}(a)$ such that\\
$$v_{p}(\frac{ \bfa_{\beta}(\mathcal{E}_{\lambda',u}, \mathfrak{c}(a))}{\log_{p}(1+p)} )= \mu_{p}'(\lam).$$
In particular, $$\mu(L_{\Sigma, \lambda}^{'}) \leq \mu_{p}'(\lam).$$
\end{prop}
\begin{proof} We basically modify the strategy in \cite[proof of Prop. $6.7$]{Hs2} 
to our setting.\\
\\
For $v=v_{\mathfrak{q}_1}$ or $v|\mathfrak{C}^-$, let $\eta_{v}$ be as in Lemma 3.6 and 3.7, respectively. We extend 
$(\eta_{v})_{v=v_{\mathfrak{q}_1}, v|\mathfrak{C^{-}}}$ to an idele $\eta = (\eta_{v})$
in $\bf A_{\cF}^{\times}$ such that\\
\begin{itemize}
\item $\eta_v=1$ for all split $v\ndivide p$,
\item $ W(\lambda_{v}^{*})\tau_{\cK_{v}/\cF_{v}}(\eta_{v}) = \lambda_{v}^{*}(\xi)$
for every finite place $v \neq v_{\mathfrak{q}_1}$.\\
\end{itemize}
The existence of $\eta$ indeed follows from Lemma 3.1.\\
\\
Recall, 
$$W(\lambda_{\sigma}^{*}) = \lambda_{\sigma}^{*}(\xi)$$ for $\sigma \in \Sigma$. From $W(\lambda^{*}) = -1$,
we thus conclude that $\tau_{\cK/\cF}(\eta) = 1$. 
In particular, $\eta$ can be written as $\beta N_{\cK/\cF}(a)$ for some $\beta \in \cF_{+}$ and $a \in \bf{A_{\cK}^{\times}}$.
By the approximation theorem, $a$ can be chosen so that $ a \equiv 1 (\mod {p(\mathfrak{q}_1 \mathfrak{C^{-}})^{n}})$ for sufficiently large $n$.\\
\\
Summarising, for every sufficiently small $\epsilon$, we have $\beta \in \cF^{\times}_{+} \cap O_{(p\mathfrak{F}\mathfrak{F^{c}})}$ such that\\
\begin{itemize}
\item $ |\beta - \eta_{v}| < \epsilon$ for all $v=v_{\mathfrak{q}_1}$ and $v$ dividing $\mathfrak{C^{-}}$,
\item $ W(\lambda_{\mathfrak{q}_{1}}^{*})\tau_{\cK_{\mathfrak{q}_{1}}/\cF_{\mathfrak{q}_{1}}}(\beta) = (-1)\cdot\lambda_{\mathfrak{q}_{1}}^{*}(\xi)$ and 
$W(\lambda_{v}^{*})\tau_{\cK_{v}/\cF_{v}}(\beta) = \lambda_{v}^{*}(\xi)$ for every finite place $v \neq v_{\mathfrak{q}_1}$.\\
\end{itemize}
We now choose $\epsilon$ small enough so that 
$A_{\beta,v}(\lam) = A_{\eta_{v},v}(\lam)$ 
for all $v|\mathfrak C^{-}$
and $\bfa_{\beta,\mathfrak{q}_1}'({\bf{c}}_{\mathfrak{q_1}}^{0},\lam) = \bfa_{\eta_{\mathfrak{q}_1}}'({\bf{c}}_{\mathfrak{q_1}}^{0},\lam) $.\\
\\
Let $\mathfrak {J} = \mathfrak{q}_{1}^{v_{\mathfrak{q}_1}(\beta)+v_{\mathfrak{q}_1}(\frakc(\cO_\cK))} \cdot \prod_{\mathfrak{q}|\mathfrak{C}} \mathfrak{q}^{v_{\mathfrak{q}}(\beta)}.$
From Lemma 3.1, it follows that $v(\beta) \equiv v(\mathfrak{c}(\cO_\cK)) ( \mod {2} )$ for every inert place 
$v \ndivide \mathfrak{q}_1\mathfrak{C^{-}}$.
Thus, there exists a fractional ideal $\mathfrak{a}$ of $\cO_\cK$ such that\\
\beq\label{E}\mathfrak{J} = (\beta) \mathfrak{c}(\cO_\cK) N_{\cK/\cF}(\mathfrak{a})^{-1}= (\beta) \mathfrak{c}(\mathfrak{a}).\eeq
\noindent\\
We now define ${\bf{c}} \in \A_{\cF,f}^{\times}$ by ${\bf {c}}_{v} = \beta^{-1}$ if $v$ is prime to $p\mathfrak{q}_{1}\mathfrak{C}\mathfrak{C^{c}}$, 
${\bf{c}}_{v_{\mathfrak{q}_1}}={\bf{c}}_{\mathfrak{q}_{1}}^{0}$ and
${\bf {c}}_{v}=1$ otherwise. Thus, $\bf \mathfrak{il}_{\cF}(c)= \mathfrak{c}(\mathfrak{a})$. Let $u \in \mathcal{U}_{p}$ such that $u \equiv \beta$ ( mod $p$).\\
\\
From Proposition 3.2, the Fourier coefficient formula and the Leibnitz product rule, we have\\
$$\bfa_{\beta}(\mathcal{E}_{\lambda',u}, \mathfrak{c}(a))=\beta^{(k-1)\Sg}\prod_{w|\Csplit}\lam_{w}(\beta)\prod_{w\in\Sg_p}\lam_w(\beta)
\cdot \bfa_{\beta,\mathfrak{q}_1}'({\bf{c}}_{\mathfrak{q_1}}^{0},\lam) \prod_{v\neq v_{\mathfrak{q}_1},v\ndivide \infty\frakD}\bfa_{\beta,v}(\frakc_v,\lam) \prod_{v|\frakC^- D_{\cK/\cF}}A_{\beta,v}(\lam)$$
\beq
= C_{\beta} \prod_{ w | \mathfrak{F}} \lambda_{w}(\beta) \cdot \bfa_{\beta,\mathfrak{q}_1}'({\bf{c}}_{\mathfrak{q_1}}^{0},\lam) 
 \prod_{v|\frakC^-}A_{\beta,v}(\lam).\\
\eeq
From our choice $p\ndivide C_{\beta}$. We are thus done by part (2) of Lemma 3.6 and 3.7.\\
\end{proof}
\noindent \\
\subsection{Upper bound II}
In this subsection, we prove an upper bound 
$$\mu(L_{\Sigma, \lambda}^{'}) \leq \mu_{p,v}'(\lam)$$
of the equality asserted in Theorem A. This subsection is quite similar to the previous subsection.\\
\\
We again start with a couple of local lemmas.\\
\begin{lm} Let $v$ be a place dividing $\mathfrak{C^{-}}$ such that $w(\mathfrak{C^{-}})=1$. 
There exists an $\eta_{v} \in \cF_{v}^{\times}$ such that 
the following conditions hold:\\
(1). $$ v_{p} ( \frac{ A_{\eta_v}'(\lam)}{\log_{p}(1+p)}) = v_{p}(\frac {\log_{p}(|\varpi_{v}|)}{\log_{p}(1+p)}),$$
(2). $$ W(\lambda_{v}^{*})\tau_{K_{v}/F_{v}}(\eta_{v}) = (-1) \cdot \lambda_{v}^{*}(\xi).$$
\end{lm}
\begin{proof} We first consider the case that $v$ is ramified.\\
\\
We recall that part (2) just puts a condition on the parity of $v(\eta_{v})$ (cf. Lemma 3.1). 
We start with an $\eta_{v}$ satisfying this condition
along with $v(2\eta_{v}) \geq -1$. From \cite[Lemma 6.1]{Hs1}, $A_{\eta_v}(\lam)=0$. Now by the Leibnitz product rule 
$$A_{\eta_v}'(\lam)=L(0,\lam_v)\widetilde{A}_{\eta_v}'(\lam).$$ 
Thus, it suffices to find $\eta_v$ satisfying (1) where $A_{\eta_v}'(\lam)$ is replaced by $\widetilde{A}_{\eta_v}'(\lam)$.\\
\\
We now have the following formula from \cite[part (1) of Prop. 4.4]{Hs2}:\\
\beq 
\widetilde{A}_{\eta_{v}}(\lam N^{k})= 
(\lambda N^{k})^{*}(\varpi_{v}^{-1})|\varpi_{v}|^{1/2} + 
(\lambda N^{k})^{*}(-2\eta_{v} d_{\cF_v}^{-1} ) \epsilon ( 1, (\lambda N^{k})_{+} |\cdot|^{-1}, \psi).\eeq
\\
For the definition of local $\epsilon$-factor $\epsilon ( 1, (\lambda N^{k})_{+} |\cdot|^{-1}, \psi)$, we refer to \cite[\S3.4 and \S3.6]{Ta} and \cite[\S4.3.1]{Hs2}. 
We use the formula arising from \cite[(3.6.11)]{Ta}.\\
\\
From the Leibnitz product rule, we have
\beq \widetilde{A}_{\eta_{v}}'(\lam)
= ( \lambda^{'})^{*}(\varpi_{v}^{-1})|\varpi_{v}|^{1/2} + 
(\lambda^{'})^{*}(-2\eta_{v} d_{\cF_v}^{-1} ) \epsilon ( 1, \lambda_{+} |\cdot|^{-1}, \psi) + \lambda^{*}(-2\eta_{v} d_{\cF_v}^{-1} ) 
\epsilon ( 1, \lambda^{'}_{+} |\cdot|^{-1}, \psi)         ).\eeq
\\
Here $\epsilon ( 1, \lambda^{'}_{+} |\cdot|^{-1}, \psi)$ is the $p$-adic 
derivative of  $\epsilon ( 1, (\lambda N^{k})_{+} |\cdot|^{-1}, \psi)$, for example via the formula \cite[(3.6.11)]{Ta}. 
From the formula, $\log_p(|\varpi_v|)$ divides $\epsilon ( 1, \lambda^{'}_{+} |\cdot|^{-1}, \psi)$. Say $\epsilon ( 1, \lambda^{'}_{+} |\cdot|^{-1}, \psi)=\log_p(|\varpi_v|)b$.\\
\\
We recall
$\lambda^{*}((-2\eta_{v} d_{\cF_v}^{-1} )) = \tau_{\cK_{v}/\cF_{v}}((-2\eta_{v} d_{\cF_v}^{-1} ))$ (cf. (3.1)). This values is already determined by part (2). For simplicity, suppose that it equals one.\\
\\
We try $\eta_v$ of the form $(-2)^{-1}d_{\cF_v}\varpi_{v}^m$, for $m\in \Z$. Let $a=\epsilon ( 1, \lambda_{+} |\cdot|^{-1}, \psi)$. In view of part (2) of Lemma 3.4, we thus have
\beq 
\widetilde{A}_{\eta_{v}}'(\lam)=\log_p(|\varpi_v|)(|\varpi_v|^{1/2}+(-1)^{m}ma+b). 
\eeq
We immediately note that there exists $m\in \Z$ of given parity such that $p \ndivide |\varpi_v|^{1/2}+(-1)^{m}ma+b$. This finishes the proof of the ramified case.\\
\\
We now consider the inert case.\\
\\
The only change in this case 
is the formula for $\widetilde{A}_{\eta_{v}}(\lam N^{k})$.
We first recall that 
$\lambda N^{k}|_{O_{v}^{\times}} = 1$.\\ 
\\
Let $\eta_v \in O_v$. Thus, from \cite[part (3) of Prop. $4.5$]{Hs2}
\beq 
\widetilde{A}_{\eta_{v}}(\lam N^{k}) =-|\varpi_{v}| + \sum_{j=0}^{v(2\eta_v)} (\lambda N^{k})^{*}(\varpi_{v}^{j}) 
(1 -|\varpi_{v}|) - (\lambda N^{k})^{*}(\varpi_{v}^{v(2\eta_v)+1})|\varpi_{v}|,
\eeq
(cf. [loc. cit., (4.16)]).\\
\\
From the formula for $\lam'^{*}$ (cf. part (2) of Lemma 3.4), $\sum_{}^{}$ expression is quite similar to (3.11). 
As the argument is very similar to the proof of Lemma 3.6, we skip the details.\\
\end{proof}
\noindent \\
Summing up these preparations, we can prove the following proposition.\\
\begin{prop} 
$$\mu(L_{\Sigma,\lam}^{'}) \leq \min_{v|\mathfrak{C^{-}}} \{\mu_{p}'(\lam), \mu_{p,v}'(\lam_{v})\}.$$
\end{prop}
\begin{proof} If $\mu_{p}(\lam_{v})=0$ for all $v|\mathfrak{C}^{-}$, then the proposition follows from 
Proposition 3.8. Thus, we suppose that $\mu_{p}(\lam_{v_{1}}) \neq 0$ for some $v_{1} | \mathfrak{C^{-}}$.\\
\\
In this case, $w_{1}(\mathfrak{C^{-}})=1$ (cf. \cite[proof of Prop. 6.3]{Hs1}). Thus, we are in the situation of the last two lemmas.\\
\\
Let $\eta_{v_{1}}$ be as in these lemmas depending on whether $v_1$ is ramified or inert.
Let $\eta_{v}$ for $v|\mathfrak{C^{-}}$ and $v \neq v_{1}$ be as in Lemma 3.9.\\
\\
Extend $(\eta_{v})_{v|\mathfrak{C^{-}}}$ to an idele $(\eta_{v})$ in $\bf A_{\cF}^{\times}$ in the same way as in the proof of Proposition 4.8.
Proceeding as in the same proof, we get $\beta \in \cF_{+}$, $u \in \cD_0$ and $\mathfrak{c}(a)$ such that\\
$$v_{p}(\frac{\bfa_{\beta}(\mathcal{E}_{\lambda',u}, \mathfrak{c}(a))}{\log_{p}(1+p)} )= \mu_{p,v_1}'(\lam).$$
From Theorem 2.2, this finishes the proof.\\
\end{proof}
\noindent\\
\begin{cor}
Theorem A holds.
\end{cor}
\begin{proof}
 This follows from Proposition 3.5 and 3.10.\\
\end{proof}
\noindent\\
\section{Non-vanishing of anticyclotomic regulator}
\noindent In this section, we consider the case $\cF=\Q$. We prove the non-vanishing of the anticyclotomic regulator 
of a self-dual CM modular form with root number $-1$.\\
\\
Let the notation and hypothesis be as in Theorem A. Let $f_{\lam}$ be 
the CM modular form associated to $\lam$ i.e. an elliptic modular form with $q$-expansion $f_{\lam}(q)$ given by
\beq
f_{\lam}(q)=\sum_{\mathfrak{a}}\lam(\mathfrak{a})q^{N(\mathfrak{a})},
\eeq
where the sum is over integral ideals of $\cK$.\\
\\
We first introduce some notation. Let $\mathfrak{p}$ be a prime above $p$ in $\cK$, $\cK_{f_{\lam},\mathfrak{p}}$ be the corresponding Hecke field and $\cO$ the corresponding ring of integers.
Let $T$ be the $\mathfrak{p}$-adic Galois representation associated to $f_{\lam}$, $V= T \otimes \cK_{f_{\lam},\mathfrak{p}}$, $W=V/T$(cf. \cite[\S1.1]{A}).
Let $\Sel(\cK_{\infty}^{-},?)$ be the anticyclotomic Selmer group in [loc. cit.,\S1.3]) for $?=T, V, W$ or the Tate duals. Let $\cX(\cK_{\infty}^{-})$ and $\cX^{*}(\cK_{\infty}^{-})$
be the dual Selmer groups in [loc. cit.,\S1.3]) for $?=W$ and the Tate dual $W^*$, respectively. Let 
\beq
h_{\infty}: (\Sel(\cK_{\infty}^{-},T) \widehat{\otimes} \bar{\Q}_p) \otimes_{\bar{\Z}_{p}\powerseries{\Gamma^{-}}} (\Sel(\cK_{\infty}^{-},T^{*})^{\iota} \widehat{\otimes} \bar{\Q}_p) 
\longrightarrow \bar{\Z}_{p}\powerseries{\Gamma^{-}}
\eeq
be the anticyclotomic height pairing in [loc. cit.,\S4.4]. The notation $\Sel(\cK_{\infty}^{-},T^{*})^{\iota}$ is as in [loc. cit., \S1.3]. 
The anticyclotomic regulator $\cR_\lam$ is defined to be the characteristic ideal of the cokernel of $h_{\infty}$.\\
\begin{prop} Suppose that $p\ndivide h_{\cK}$. Then, the anticyclotomic regulator $\mathcal{R}_{\lam}$ does not vanish.
\end{prop}
\begin{proof}
From [loc. cit., Thm. 2.2], $\cX^{*}(\cK_{\infty}^{-})$ is a $\cO \powerseries{\Gamma^{-}}$-module of rank one. 
Let $\cX \subset \cO \powerseries{\Gamma^{-}}$ be the characteristic ideal of the torsion sub-module of $\cX^{*}(\cK_{\infty}^{-})$.\\
\\
From \cite[Thm. 2.2]{A},\\
\beq \cX \cR_{\lam}= (L_{\Sigma,\lam}^{'}) \eeq
as ideals of $\Zbarp \powerseries{\Gamma^{-}} \otimes_{\Zp} \Qp$.\\
\\
The proposition thus follows by Theorem A.\\
\noindent\\
\end{proof}
\begin{remark} When $\lam$ is a Gr\"{o}ssencharacter of an elliptic curve over $\Q$ with CM by $\cO_\cK$,
the above proposition is proven via Iwasawa theory of CM elliptic curves and a non-vanishing 
result of Rohrlich (cf. \cite[App.]{AH}).\\
\end{remark}



\end{document}